\documentclass[a4paper, draft]{amsart}

\usepackage{amsmath, amssymb, dsfont, tikz} 
\usetikzlibrary{shapes}
\usetikzlibrary{calc}
\usetikzlibrary{arrows}

\theoremstyle{plain}
\newtheorem{lemma}{Lemma}[section]
\newtheorem{definition}{Definition}[section]
\newtheorem{theorem}[lemma]{Theorem}

\newtheorem{conjecture}[lemma]{Conjecture}

\theoremstyle{remark}
\newtheorem{remark}[lemma]{Remark}

\newcommand{\Indicator}[1]{\mathds{1}\left(#1\right)}

\title[Mobility can drastically improve the heavy traffic performance]
{Mobility can drastically improve the heavy traffic performance from $\frac{1}{1-\varrho}$ to $-\log(1-\varrho)$}

\author{Florian Simatos}
\address{ISAE SUPAERO and Universit\'e de Toulouse}
\email{florian.simatos@isae.fr}

\author{Alain Simonian}
\address{ORANGE LABS}
\email{alain.simonian@orange.com}

\date{\today}

\numberwithin{equation}{section}

\newcommand{\f}{\text{\textnormal{\texttt{f}}}}
\newcommand{\net}{\text{\textnormal{\texttt{net}}}}
\newcommand{\tot}{\text{\textnormal{\texttt{tot}}}}

\newcommand{\cF}{\mathcal{F}}
\renewcommand{\d}{\texttt{d}}
\newcommand{\E}{\mathbb{E}}
\newcommand{\N}{\mathbb{N}}
\renewcommand{\P}{\mathbb{P}}
\newcommand{\R}{\mathbb{R}}

\begin{document}

\begin{abstract}
We study a model of wireless networks where users move at speed $\theta \geq 0$, which has the original feature of being defined through a fixed-point equation. Namely, we start from a two-class Processor-Sharing queue to model one representative cell of this network: class $1$ users are not impatient (non-moving) and class $2$ users are impatient (moving). This model has five parameters, and we study the case where one of these parameters is set as a function of the other four through a fixed-point equation. This fixed-point equation captures the fact that the considered cell is in balance with the rest of the network. This modeling approach allows us to alleviate some drawbacks of earlier models of mobile networks.

Our main and surprising finding is that for this model, mobility drastically improves the heavy traffic behavior, going from the usual $\frac{1}{1-\varrho}$ scaling without mobility (i.e., when $\theta = 0$) to a logarithmic scaling $- \log(1-\varrho)$ as soon as $\theta > 0$. In the high load regime, this confirms that the performance of mobile system takes benefit from the spatial mobility of users. Other model extensions and complementary methodological approaches to this heavy traffic analysis are finally discussed.
\end{abstract}

\maketitle

\setcounter{tocdepth}{1}
\tableofcontents


\section{Introduction}


\subsection{Background and undesirable ergodicity assumption}

Since the emergence of wireless networks and following their continual development, the impact of user mobility on network performance has attracted significant attention. In~\cite{Grossglauser01:0}, the authors showed that mobility creates a multi-user diversity leading to a significant improvement in per-user throughput. Since this seminal work, the observation that mobility increases throughput has been confirmed in a wide variety of situations captured by various stochastic models, see 
~\cite{Baynat15:0, Bonald04:0, Bonald09:1, Borst09:0, Borst06:0, Borst13:0, Ma14:0, Simatos10:0}. Interestingly, to the best of our knowledge, the first paper to show that mobility could under certain circumstances actually degrade delay only appeared recently~\cite{Anton19:0}.

In all these models, user mobility is represented by an ergodic process on a finite region of the plane. For instance, users follow 
in~\cite{Grossglauser01:0} a stationary and ergodic trajectory on the unit disk; in~\cite{Bonald04:0, Bonald09:1, Borst09:0, Borst06:0}, users follow an irreducible Markovian trajectory in a network consisting of a finite number of cells. In our view, one of the limitations of such a modeling assumption is the highly unrealistic behavior it displays under congestion. Indeed, in the congestion regime, users stay in the network for a long time, so that if their trajectory is ergodic, they necessarily visit the same place a large number of times, as if they were walking circularly.

\subsection{High-level model description and motivation}
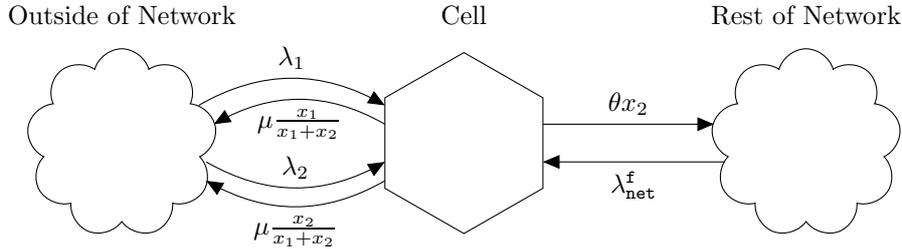
\begin{figure}
	\centering
		\begin{tikzpicture}
			\begin{scope}[xshift=0cm]
				\node [cloud, cloud puffs=9, draw, minimum width=2.5cm, minimum height=2.5cm] (out) {};
				\node at (0,17mm) {Outside of Network};
			\end{scope}
			\begin{scope}[xshift=45mm]
				\draw (90:1.2cm) -- (150:1.2cm) -- (210:1.2cm) coordinate [midway] (cellW) -- (270:1.2cm) -- (330:1.2cm) -- (390:1.2cm) coordinate [midway] (cellE) -- (450:1.2cm);
				\node at (90:17mm) {Cell};
			\end{scope}
			\begin{scope}[xshift=90mm]
				\node [cloud, cloud puffs=9, draw, minimum width=2.5cm, minimum height=2.5cm] (net) {};
				\node at (0,17mm) {Rest of Network};
			\end{scope}
			
			\path [-triangle 45] ($(out.east)+(-.2,.5)$) edge [bend left] node [above] {$\lambda_1$} ($(cellW)+(0,.5)$);
			\path [triangle 45-] ($(out.east)+(0,.25)$) edge [bend left] node [below] {$\mu \frac{x_1}{x_1 + x_2}$} ($(cellW)+(0,.25)$);

			\path [-triangle 45] ($(out.east)+(-.1,-.25)$) edge [bend right] node [above] {$\lambda_2$} ($(cellW)+(0,-.25)$);
			\path [triangle 45-] ($(out.east)+(-.1,-.5)$) edge [bend right] node [below] {$\mu \frac{x_2}{x_1 + x_2}$} ($(cellW)+(0,-.5)$);

			\path [-triangle 45] ($(cellE)+(0,.25)$) edge node [above] {$\theta x_2$} ($(net.west)+(0,.25)$);
			\path [triangle 45-] ($(cellE)+(0,-.25)$) edge node [below] {$\lambda^\f_\net$} ($(net.west)+(.14,-.25)$);
		\end{tikzpicture}
	\caption{Description of the model considered in the paper. Without imposing the balance condition corresponding to the fixed-point equation~\eqref{eq:FP}, this is a two-class Processor-Sharing queue with one impatient class, namely the class-$2$ of mobile users, with arrival rate $\lambda_2 + \lambda^\f_\net = \lambda^\f_\tot$. The balance equation~\eqref{eq:FP} accounts for the fact that a typical cell at equilibrium is considered, with equal flows from and to the rest of the network.}
\label{fig:model}
\end{figure}

In the present paper, we pursue the modeling approach started 
in~\cite{Olivier19:0, Simonian16:0}. The main idea to alleviate the aforementioned drawback resulting from the ergodic trajectory assumption is to focus on a single cell and abstract the rest of the network as a single state. By doing so, we only keep track of the precise location of users when they are located in the considered cell: when located elsewhere (either outside the network or in the rest of the 
network), we do not track them precisely. This simple model could be generalized by focusing on several cells rather than a single one (see the discussion in Section~\ref{sec:extensions} below). Users can thus be in one of three ``places'', as pictured in Figure~\ref{fig:model}:
\begin{enumerate}
	\item outside the network, meaning that they do not require service (the left fluffy shape);
	\item in the considered cell (the middle hexagon);
	\item in the network but not in the considered cell, i.e., in the rest of the network (the right fluffy shape).
\end{enumerate}


Moreover, our work is motivated by future LTE networks where cells can be small in range (pico, femto cells). In this case, users experience similar radio conditions and we will therefore assume below that they receive the same transmission capacity, independently of her location within the cell. While focusing on the spatial mobility aspect of users, the present study consequently ignores the possible spatial variations of transmission capacity inevitably presented by larger cells. In the following, this equal capacity is denoted by $1/\mu$.

\subsection{Mathematical model and results}
Our mathematical model is introduced in two steps. At this stage, we only give a high-level description of our model in order to give the big picture: details are provided in Section~\ref{sec:model}.

We first introduce a ``free'' model $\mathbf{X}^\f$, which is simply a two-class Processor-Sharing queue with one impatient class: from the mobile network perspective, non-impatient users correspond to static users who do not move, and impatient users to mobile users who move and thus potentially leave the cell to the rest of the network. The non-zero transition rates of Markov process $\mathbf{X}^\f$ are given by
\begin{equation} \label{eq:transition-rates-free}
	\mathbf{x} \in \N^2 \longrightarrow \begin{cases}
		\mathbf{x}+\mathbf{e}_1 & \text{ at rate } \lambda_1,
		\\ \\
		\mathbf{x}+\mathbf{e}_2 & \text{ at rate } \lambda_2 + \lambda^\f_\net,
		\\ \\
		\mathbf{x}-\mathbf{e}_1 & \text{ at rate } \displaystyle{\mu \, \frac{x_1}{x_1 + x_2}},
		\\ \\
		\mathbf{x}-\mathbf{e}_2 & \text{ at rate } 
		\displaystyle{\mu \, \frac{x_2}{x_1 + x_2} + \theta x_2},
	\end{cases}
\end{equation}
with $\mathbf{e}_1 = (1,0)$ and $\mathbf{e}_2 = (0,1)$ (see 
Section~\ref{sub:free} for a detailed interpretation of these parameters); as specified below, $\theta$ represents the impatience/mobility rate.

In a second step, we introduce our full model which is obtained from the free model~\eqref{eq:transition-rates-free} by enforcing a balance condition in the form of the fixed-point equation~\eqref{eq:FP} detailed below. This fixed-point equation means that the flows of mobile users to and from the rest of the network must balance out. This condition consequently means that the considered cell is ``typical'', in that the cell imposes a load on the rest of the network equal to the reciprocal load from the rest of the network to the considered cell.

If $\varrho_1 = \lambda_1/\mu$ denotes the load of static (i.e., non-impatient) users and $\varrho_2 = \lambda_2 / \mu$ the load of mobile 
(i.e., impatient) users, the stability condition without enforcing this balance equation is $\varrho_1 < 1$ since class-$2$ users are impatient and thus cannot accumulate (see Lemma~\ref{lemma:stab-f}). From the mobile network perspective, the interpretation is that mobile users can always escape to the rest of the network where they are not tracked. The stability condition $\varrho_1 < 1$ is therefore clearly fictitious, because even if we do not keep track of the precise location of mobile users in the rest of the network, they still impose a load on the network which should be accounted for. When enforcing the balance equation~\eqref{eq:FP}, the stability condition then becomes $\varrho_1 + \varrho_2 < 1$ which is the natural expected stability condition since, considering the cell as a representative cell of a larger network, $\varrho_1 + \varrho_2$ is the normalized load per cell (see Lemma~\ref{lemma:sol-FP}).

The study of this model is driven by the desire to understand the impact of mobility on performance. We wish, in particular, to address questions such as: given the total load $\varrho = \varrho_1 + \varrho_2 < 1$, does the network perform better if the proportion $\varrho_2 / \varrho$ of mobile users increases ? Answering such a question being generally difficult, we here resort to the approximation obtained in the heavy traffic regime where $\varrho \uparrow 1$. In addition to providing useful insight into the impact of mobility on performance, this model turns out to exhibit a highly original heavy traffic behavior, whereby the number of users in system scales like 
$-\log(1-\varrho)$ as $\varrho \uparrow 1$. If all users were static we would have the usual $(1-\varrho)^{-1}$ scaling; our model therefore suggests that not only throughput but also delay is improved with mobility.

To the best of our knowledge, this unusual heavy traffic scaling only appeared earlier \cite{Lin10:0} in the case of the Shortest-Remaining-Processing-Time service discipling with heavy tails service distribution. In this case, such an improvement is conceivable: indeed, since the service distribution is heavy tailed, very long jobs are not so rare. If the service discipline is FIFO, then these jobs impose a very large delay on the numerous smaller jobs that arrive after them. With SRPT, in contrast, only the large jobs spend a long time in the network, essentially due to their large service requirement. As regards the impact of mobility in wireless networks, it has been already observed~\cite{Simonian16:0}, through an approximate analysis and extensive simulation, that the performance gain due to mobility can be related to an ``opportunistic'' displacement of mobile users within the network; in fact, any local increase of traffic in one given cell induces the displacement of the moving users to a neighboring cell in order to complete their transmission, hence alleviating the traffic for remaining (static or moving) users in the original cell. Our contribution in this paper is to theoretically justify this statistical behavior in the heavy traffic regime.

\subsection{Organization of the paper}
We start by introducing our model and Theorem~\ref{thm:main}, the main result of the paper, in Section~\ref{sec:model}. In this section, we will also present a conjecture refining our main result, which is discussed in Section~\ref{sec:extensions}. Section~\ref{sec:proof-HT-f} and~\ref{sec:proof-2} are devoted to the proof of Theorem~\ref{thm:main}.


\section{Model description and main result} \label{sec:model}

We now introduce our model in details: as above, we first address a ``free'' model simply represented by a two-class Processor-Sharing queue with one impatient class; further, we introduce the full model which derives from the free model by enforcing a balance condition in the form of a fixed-point equation~\eqref{eq:FP}. We then state our main result and explain the main steps of the proof.

\subsection{Free model} \label{sub:free}
In the free model represented by the Markov process $\mathbf{X}^\f$, with non-zero transition rates~\eqref{eq:transition-rates-free}, we consider two classes of users:

\textbf{1)} class-$1$ users are static: they arrive to the cell from the outside at rate $\lambda_1$, require a service which is exponentially distributed with parameter $\mu$ and are served according to the Processor-Sharing service discipline. They consequently leave the network (to the 
outside) at an aggregate rate $\mu x_1 / (x_1 + x_2)$, with $x_i$ the number of class-$i$ users;

\textbf{2)} class-$2$ users are mobile: they arrive to the cell from the outside at rate $\lambda_2$, require a service which is exponentially distributed with parameter $\mu$ and are served according to the Processor-Sharing service discipline. As for class-$1$ users, they leave the network to the outside upon completing service at an aggregate rate 
$\mu x_2 / (x_1 + x_2)$; the difference with class-$1$ users is that they are mobile and can thus leave the cell (now, to the rest of the network and not the outside) before completing service. We assume that each mobile user leaves the cell at rate~$\theta$, and so class-$2$ users leave the cell to the rest of the network at an aggregate rate $\theta x_2$. Finally, mobility can also make users enter the cell from outside the network and we assume that this happens at rate $\lambda^\f_\net$.


At this stage, it is apparent from rates~\eqref{eq:transition-rates-free} that differentiating the outside and the rest of the network is artificial and bears no consequence on the distribution of this Markov process. All that matters is the total arrival rate 
$\lambda^\f_\tot := \lambda_2 + \lambda^\f_\net$ and the total service rate $\mu x_2/(x_1 + x_2) + \theta x_2$ of class-$2$ users. This distinction, however, will become crucial later.

The distribution of Markov process $\mathbf{X}^\f$ with non-zero transition rates \eqref{eq:transition-rates-free} thus depends on the five parameters 
$\lambda_1$, $\lambda_2$, $\lambda^\f_\net$, $\theta$ and $\mu$ (and more precisely, on $\lambda_2$ and $\lambda^\f_\net$ only through their sum 
$\lambda^\f_\tot = \lambda_2 + \lambda^\f_\net$). The superscript \f\ refers to ``free'', as the ``full'' process in that we will be mainly interested belongs to this class, but with $\lambda^\f_\net$ chosen as a function of the other four parameters $\lambda_1$, $\lambda_2$, $\theta$ and $\mu$.

In the rest of the paper, we write $\varrho_i = \lambda_i / \mu$ and 
$\varrho = \varrho_1 + \varrho_2$. The following result describes the stability region of $\mathbf{X}^\f$, which depends on whether 
$\theta = 0$ or $\theta > 0$. Whenever $\mathbf{X}^\f$ is positive recurrent, we denote by $\mathbf{X}^\f(\infty)$ its stationary distribution. Here and throughout the paper, vector inequalities are understood component-wise, so for instance $\E(\mathbf{X}^\f(\infty)) < \infty$ means that $\E(X^\f_i(\infty)) < \infty$ for $i \in \{1, 2\}$.

\begin{lemma} \label{lemma:stab-f}
	Stability of $\mathbf{X}^\f$ depends on whether $\theta = 0$ or $\theta > 0$ in the following way:
	\begin{itemize}
		\item if $\theta = 0$, then $\mathbf{X}^\f$ is positive recurrent if 
		$\varrho + \lambda^\f_\net / \mu < 1$, null recurrent if 
		$\varrho + \lambda^\f_\net / \mu = 1$ and transient if 
		$\varrho + \lambda^\f_\net / \mu > 1$;
		\item if $\theta > 0$, then $\mathbf{X}^\f$ is positive recurrent if 
		$\varrho_1 < 1$, null recurrent if $\varrho_1 = 1$ and transient if 
		$\varrho_1 > 1$. 
	\end{itemize}
In either case, when the process is positive recurrent, then we have 
$\E(\mathbf{X}^\f(\infty)) < \infty$.
\end{lemma}

\noindent
These results can be proved with Lyapounov-type arguments and the comparison with suitable $M/M/1$ queues. Such arguments are standard and the proof is therefore omitted.

\subsection{Constrained model}
The previous result formalizes the behavior pointed out in the introduction, namely that in the presence of mobile users (i.e., when $\theta > 0$), mobile users do not matter as regards to stability. In fact, if they accumulate, they can then escape to the rest of the network where they are not tracked. However, this is only an artifact of our modeling approach since mobile users that escape to the rest of the network should somehow be accounted for. The goal of the constrained model $\mathbf{X}$ that we now introduce aims at doing this; it is obtained by taking 
$\lambda^\f_\net$ as a function of the other four parameters through a 
fixed-point equation.

\subsubsection{The fixed-point equation}

In the free model, the three parameters $\lambda_1$, $\lambda_2$ and 
$\mu$ govern the transition involving the outside, while the two parameters 
$\theta$ and $\lambda^\f_\net$ govern transitions within the network. Out of these five parameters, all but $\lambda^\f_\net$ can be considered as exogenous and dictated by the users' behavior: how often do they arrive, how fast they move, etc. In contrast, $\lambda^\f_\net$ is hard to directly tie down with users' behavior and is more an artifact of our modeling approach.

In order to fix the value of $\lambda^\f_\net$ in an exogenous way, the idea is to impose a balance condition. Roughly speaking, we assume that the cell is in equilibrium (see Section~\ref{sec:extensions} for a discussion on this assumption) and that the flows of mobile users to and from the rest of the network balance each other. Provided that $\mathbf{X}^\f$ is positive recurrent, we thus want to impose the balance equation
\begin{equation} \label{eq:FP}
	\lambda^\f_\net = \theta \cdot \E \left( X^\f_2(\infty) \right). 
	\tag{FP}
\end{equation}
We note that (\ref{eq:FP}) is a fixed-point equation, as 
$\E(X^\f_2(\infty))$ is a function of $\lambda^\f_\net$, the other four parameters being kept fixed. Provided that there exists a unique solution to~\eqref{eq:FP} with the four parameters $\lambda_1, \lambda_2, \mu$ and $\theta$ given (necessary and sufficient conditions for this will be stated below), this unique solution is denoted by $\Lambda_\net$. We then consider the process $\mathbf{X}$ with the same transition rates~\eqref{eq:transition-rates-free} than the free process, but where the value of parameter $\lambda^\f_\net$ has been set to $\Lambda_\net$, chosen as a function of 
$\lambda_1, \lambda_2, \mu$ and $\theta$ 
via~\eqref{eq:FP}. The process $\mathbf{X}$ will be the main object of investigation in this paper.

\begin{definition}
	Provided that there exists a unique solution to~\eqref{eq:FP}, denoted $\Lambda_\net = \Lambda_\net(\lambda_1, \lambda_2, \mu, \theta)$, the constrained model $\mathbf{X}$ is the $\N^2$-valued Markov process with non-zero transition rates given by~\eqref{eq:transition-rates-free} with $\lambda^\f_\net = \Lambda_\net$.
\end{definition}

Our main result is that even a slight amount of mobility (i.e., $\theta > 0$ even very small, instead of $\theta = 0$) dramatically increases the performance of the network and leads to a unusual $-\log(1-\varrho)$ heavy traffic scaling. To explain this we first discuss the case 
$\theta = 0$ with no mobility.

\subsubsection{Heavy traffic regime}
When we say $\varrho \uparrow 1$, we mean that we consider a sequence of systems indexed by $n$, where the parameters $\lambda^n_1$, $\lambda^n_2$, 
$\mu^n$ and $\theta^n$ in the $n$-th system satisfy $\varrho^n < 1$ (where 
$\varrho^n_i = \lambda^n_i / \mu^n$, $\varrho^n = \varrho^n_1 + \varrho^n_2$) and as $n \to \infty$, we have $\lambda^n_i \to \lambda_i$, 
$\mu^n \to \mu$, $\theta^n \to \theta$ with 
$\lambda_1, \lambda_2, \mu, \theta \in (0,\infty)$, 
$\varrho = 1$ where $\varrho = \varrho_1 + \varrho_2$ and 
$\varrho_i = \lambda_i / \mu$. We then use the notation $\Rightarrow_\varrho$ to mean weak convergence as $\varrho \uparrow 1$.

We will also consider convergence when other parameters vary. We use, in particular, the notation $\Rightarrow_{\lambda^\f_\tot}$ to mean weak convergence as $\lambda^\f_\tot \to \infty$, and also introduce another parameter $\varepsilon > 0$ and use the notation 
$\Rightarrow_{\lambda^\f_\tot, \varepsilon}$ to mean weak convergence first as $\lambda^\f_\tot \to \infty$ and then as $\varepsilon \downarrow 0$. To be more precise, $Z \Rightarrow_{\lambda^\f_\tot, \varepsilon} Z'$ means that for any continuous and bounded function $f$ we have
\[ \limsup_{\lambda^\f_\tot \to \infty} \left \lvert \E \left( f(Z) \right) - \E \left( f(Z') \right) \right \rvert \xrightarrow[\varepsilon \to 0]{} 0. \]

\subsubsection{The case $\theta = 0$}
Consider now the case $\theta = 0$. We distinguish two cases :
\begin{itemize}
	\item if $\varrho \geq 1$, then the free process is transient or null recurrent, and so~\eqref{eq:FP} is not defined;
	\item if $\varrho < 1$, $0$ is the only solution to~\eqref{eq:FP} because $\E(X^\f_2(\infty)) < \infty$ by Lemma~\ref{lemma:stab-f}.
\end{itemize}
Thus, the constrained model is only defined for $\varrho < 1$; in this case, it corresponds to the free process with $\lambda^\f_\net = 0$ and is in particular positive recurrent. The following result, taken from~\cite{Rege96:0}, states that its heavy traffic behavior obeys the usual $(1-\varrho)^{-1}$ scaling.

\begin{lemma} 
\label{lemma:HT-f-0}
If $\theta = 0$, then $(1-\varrho) \mathbf{X}(\infty) \Rightarrow (E, E)$ with $E$ an exponential random variable with parameter $2$.
\end{lemma}

\subsubsection{The case $\theta > 0$}
We now show that, whatever the value of $\theta > 0$, the behavior changes dramatically and leads to a unusual $-\log(1-\varrho)$ scaling. We first investigate the existence and uniqueness to the fixed-point 
equation~\eqref{eq:FP}. The proof relies on monotonicity and continuity arguments detailed in~\cite{Olivier19:0} and it is thus only briefly recalled here.

\begin{lemma} 
\label{lemma:sol-FP}
Assume that $\theta > 0$. If $\varrho < 1$, then there exists a unique solution to~\eqref{eq:FP}. If $\varrho_1 < 1$ but $\varrho \geq 1$, then there is no solution to~\eqref{eq:FP}.
\end{lemma}

This result is comforting: indeed, $\varrho < 1$ is the ``natural'' stability condition. Comparing Lemmas~\ref{lemma:stab-f} and~\ref{lemma:sol-FP}, we see that imposing~\eqref{eq:FP} changes the stability condition from $\varrho_1 < 1$ (mobile users do not matter) to $\varrho < 1$ (mobile users matter). Moreover, we observe the peculiar feature that, whenever the stability condition is violated, the Markov process is not defined at all, and not simply transient as is usually the case. This is due to the fact that we seek to impose a long-term balance equation through~\eqref{eq:FP}, which cannot be sustained for a system out of equilibrium.

For completeness and since the key equation~\eqref{eq:P-rho} below will be useful later, we provide a short sketch of the proof of Lemma~\ref{lemma:sol-FP}. So consider $\theta > 0$ and assume $\varrho_1 < 1$, since otherwise 
$\mathbf{X}^\f(\infty)$ is not defined. Let 
$$
Q(\lambda^\f_\net) = \P(\mathbf{X}^\f(\infty) = \mathbf{0}),
$$
the other four parameters being fixed. The balance of flow for the free system entails
$ \lambda_1 + \lambda_2 + \lambda^\f_\net = \mu \, \P(\mathbf{X}^\f(\infty) \neq \mathbf{0}) + \theta \, \E(X^\f_2(\infty))$ or equivalently,
\begin{equation} \label{eq:h}
	Q(\lambda^\f_\net) = 1 - \varrho - \frac{\lambda^\f_\net}{\mu} + 
	\frac{\theta}{\mu} \, \E(X^\f_2(\infty)).
\end{equation}
In particular,~\eqref{eq:FP} is equivalent to
\begin{equation} \label{eq:P-rho}
\P(\mathbf{X}^\f(\infty) = \mathbf{0}) = 1 - \varrho.
\end{equation}
Since $\P(\mathbf{X}^\f(\infty) = \mathbf{0}) > 0$, this relation shows that no solution can exist for $\varrho \geq 1$. Assume now that 
$\varrho < 1$. It is intuitively clear that $Q$ is continuous and strictly decreasing to $0$: as class-$2$ users arrive at a higher rate, the probability of the system being empty decreases strictly and continuously to $0$. As 
$Q(0) > 1-\varrho$ after~\eqref{eq:h}, this entails the existence and uniqueness of solutions to~\eqref{eq:FP}. We recall that this unique solution is written $\Lambda_\net$ and define 
$$
\Lambda_\tot = \lambda_2 + \Lambda_\net
$$
as the total arrival rate of class-$2$ users in the constrained model.

According to Lemma~\ref{lemma:sol-FP}, the heavy traffic behavior consists in letting $\varrho \uparrow 1$ when $\theta > 0$. The following result is the main result of the paper. Extensions of this result are discussed in Section~\ref{sec:extensions}.

\begin{theorem} 
\label{thm:main}
	Assume that $\theta > 0$. As $\varrho \uparrow 1$, the sequence
	\[ \frac{\mathbf{X}(\infty)}{-\log(1-\varrho)} \]
	is tight and any of its accumulation points is almost surely smaller than the point $\pmb{\xi}^*$ given by
	\[ 
	\pmb{\xi}^* = (\pmb{\xi}_1^*,\pmb{\xi}_2^*) = \left( \frac{\varrho_1}{1-\varrho_1}, 1 \right). 
	\]
\end{theorem}

This result shows that adding even a slight amount of mobility, i.e., going from $\theta = 0$ to $\theta > 0$, dramatically changes the heavy traffic behavior, making $\mathbf{X}(\infty)$ scale like $-\log(1-\varrho)$ instead of $1/(1-\varrho)$. We could actually show that $-\log(1-\varrho)$ is indeed the right order, i.e., accumulation points are $>0$ (see Section~\ref{sec:extensions}).

\begin{remark}
It is surprising that this upper bound does not depend on $\theta$. Indeed, when $\theta = 0$, Lemma~\ref{lemma:HT-f-0} implies that 
$\mathbf{X}(\infty) / -\log(1-\varrho) \Rightarrow_\varrho \infty$ and so interchanging limits suggests that $\mathbf{X}(\infty) / -\log(1-\varrho)$ should converge to a limit $\pmb{\xi}(\theta)$ that should blow up as $\theta \downarrow 0$. This is not the case, however, and we actually conjecture that $\mathbf{X}(\infty) / -\log(1-\varrho)$ converges to a limit independent of $\theta$ (see Section~\ref{sec:extensions}). That limits cannot be interchanged testifies from the subtlety of the result, which, we believe, is due to the fact that we need an unusual large deviation result for a two time-scale system, see Section~\ref{sub:LD}.
\end{remark}

Let us now explain where this unusual $-\log(1-\varrho)$ scaling comes from: the idea is to reduce the problem to questions on the free process 
$\mathbf{X}^\f$ by writing
\begin{equation} \label{eq:decomposition}
	\frac{\mathbf{X}(\infty)}{-\log(1-\varrho)} = 
	\frac{\Lambda_\tot}{-\log(1-\varrho)} \times 
	\frac{\mathbf{X}(\infty)}{\Lambda_\tot}.
\end{equation}
It is easy to see that $\Lambda_\tot \to \infty$ as $\varrho \uparrow 1$. Thus, as $\mathbf{X}$ is a particular case of $\mathbf{X}^\f$, understanding the asymptotic behavior of $\mathbf{X}(\infty) / \Lambda_\tot$ as $\varrho \uparrow 1$ amounts to understanding the asymptotic behavior of 
$\mathbf{X}^\f(\infty) / \lambda^\f_\tot$ as $\lambda^\f_\tot \to \infty$. The following result specifies this behavior.

\begin{lemma} \label{lemma:HT-f}
Assume that $\theta> 0$ and $\varrho_1 < 1$. Then as 
$\lambda^\f_\tot \to \infty$, the sequence 
$\mathbf{X}^\f(\infty) / \lambda^\f_\tot$ is tight and any accumulation point is almost surely smaller than the constant $\theta^{-1} \pmb{\xi}^*$ with $\pmb{\xi}^*$ given as in 
Theorem~\ref{thm:main}.
	
As $\varrho \uparrow 1$, in particular, the sequence $\mathbf{X}(\infty) / \Lambda_\tot$ is tight and any accumulation point is almost surely smaller than the constant $\theta^{-1} \pmb{\xi}^*$.
\end{lemma}

Next,~\eqref{eq:P-rho} shows that
\[ 
\frac{\Lambda_\tot}{-\log(1-\varrho)} = 
\frac{\Lambda_\tot}{-\log \P(\mathbf{X}(\infty) = \mathbf{0})} 
\]
and so, for the same reason as above, understanding the asymptotic behavior of $\Lambda_\tot / -\log(1-\varrho)$ as $\varrho \uparrow 1$ amounts to understanding the asymptotic behavior of 
$-\log \P(\mathbf{X}^\f(\infty) = \mathbf{0}) / \lambda^\f_\tot$ as 
$\lambda^\f_\tot \to \infty$.

\begin{lemma} 
\label{lemma:2}
Assume that $\theta > 0$. For any $\varrho_1 < 1$, we then have
	\[ 
	\liminf_{\lambda^\f_\tot \to \infty} \left( - \frac{1}{\lambda^\f_\tot} \log \P(\mathbf{X}^\f(\infty) = \mathbf{0}) \right) \geq 
	\frac{1}{\theta}. 
	\]
	In particular,
	\[ \limsup_{\varrho \uparrow 1} \left( \frac{\Lambda_\tot}{-\log(1-\varrho)} \right) \leq \theta. \]
\end{lemma}

In view of~\eqref{eq:decomposition}, the two previous lemmas directly imply Theorem~\ref{thm:main}. In other words, the $-\log(1-\varrho)$ scaling of $\mathbf{X}(\infty)$ arises for the two following reasons:
\begin{enumerate}
	\item the (at most) linear increase of $\mathbf{X}^\f(\infty) \leq \lambda^\f_\tot \pmb{\xi}^* + o(\lambda^\f_\tot)$ as $\lambda^\f_\tot \to \infty$;
	\item the exponential decay of 
	$\P(\mathbf{X}^\f(\infty) = \mathbf{0}) \leq 
	e^{-\lambda^\f_\tot / \theta + o(\lambda^\f_\tot)}$ as $\lambda^\f_\tot \to \infty$.
\end{enumerate}
Lemmas~\ref{lemma:HT-f} and~\ref{lemma:2} are proved in Sections~\ref{sec:proof-HT-f} and~\ref{sec:proof-2}.

\begin{remark} \label{rk:conj}
In Section~\ref{sec:extensions}, we discuss refinements of these upper bounds: in particular, we show how to prove that 
$\mathbf{X}^\f(\infty) / \lambda^\f_\tot \Rightarrow_{\lambda^\f_\tot} \theta^{-1} \pmb{\xi}^*$, and we conjecture that 
$$
\P(\mathbf{X}^\f(\infty) = \mathbf{0}) = \exp(- \kappa \lambda^\f_\tot + 
o(\lambda^\f_\tot))
$$
with constant
$$
\kappa = \frac{1 - \log(1-\varrho_1)}{\theta}.
$$
\end{remark}

\begin{remark}
The linear increase in $\lambda^\f_\tot$ of $\mathbf{X}^\f(\infty)$ is natural in the setting of single-server queues. Moreover, the refinement 
$\mathbf{X}^\f(\infty) \approx \lambda^\f_\tot \pmb{\xi}^*$ suggests that 
$\mathbf{X}^\f(\infty)$ is of the order of $\lambda^\f_\tot$. This makes state $0$ far from the typical value of $\mathbf{X}^\f(\infty)$ and the exponential decay of the stationary probability of being at $0$ is thus expected in view of the Large Deviations theory. The link with the Large Deviations theory is discussed in more details in 
Section~\ref{sec:extensions}.
\end{remark}


\section{Proof of Lemma~\ref{lemma:HT-f}} 
\label{sec:proof-HT-f}


In the rest of the paper, we use several couplings. We use the notation 
$\mathbf{X} \prec \mathbf{Y}$ to mean that we can couple $\mathbf{X}$ and 
$\mathbf{Y}$ such that $\mathbf{X} \leq \mathbf{Y}$. If 
$\mathbf{X}$ and $\mathbf{Y}$ are random processes, this is to be understood as $\mathbf{X}(t) \leq \mathbf{Y}(t)$ for all $t$, and vector inequalities are understood component-wise. 

In order to prove Lemma~\ref{lemma:HT-f}, we first exhibit a family of processes 
$\mathbf{Y}'$ indexed by some additional parameter $\varepsilon > 0$ and with $\mathbf{X}^\f \prec \mathbf{Y}'$ for every $\varepsilon > 0$, and 
$\mathbf{Y}'(\infty) / \lambda^\f_\tot \Rightarrow_{\lambda^\f_\tot, \varepsilon} \theta^{-1} \pmb{\xi}^*$. We build this coupling in two steps, and then analyze the process $\mathbf{Y}'$. In order to prove that 
$\mathbf{Y}'(\infty) / \lambda^\f_\tot \Rightarrow_{\lambda^\f_\tot, \varepsilon} \theta^{-1} \pmb{\xi}^*$, we then exhibit another family of processes 
$\mathbf{Y}$ with
$\mathbf{Y}(\infty) / \lambda^\f_\tot \Rightarrow_{\lambda^\f_\tot, \varepsilon} \theta^{-1} \pmb{\xi}^*$ and such that 
$(\mathbf{Y}(\infty) - \mathbf{Y}'(\infty)) / \lambda^\f_\tot 
\Rightarrow_{\lambda^\f_\tot, \varepsilon} 0$.

\subsection{First coupling: $\mathbf{X}^\f \prec \widetilde{\mathbf{Y}}$}

Starting from~\eqref{eq:transition-rates-free}, the first step consists in neglecting the term $\mu x_2 / (x_1 + x_2)$ in the departure rate of 
$X^\f_2$ by lower bounding it by~$0$. When we do so, this makes the departure rate smaller for the second coordinate, which makes it larger, which in turn makes the departure rate $\mu y_1 / (y_1 + y_2)$ from the first coordinate smaller, and hence the first coordinate larger. Thus if 
$\widetilde{\mathbf{Y}}$ is the $\N^2$-valued Markov process with non-zero transition rates
\[ \mathbf{y} \in \N^2 \longrightarrow \begin{cases}
	\mathbf{y}+\mathbf{e}_1 & \text{ at rate } \lambda_1,\\
	\mathbf{y}+\mathbf{e}_2 & \text{ at rate } \lambda^\f_\tot,\\
	\mathbf{y}-\mathbf{e}_1 & \text{ at rate } \mu y_1 / (y_1 + y_2),\\
	\mathbf{y}-\mathbf{e}_2 & \text{ at rate } \theta y_2,
\end{cases} \]
then we have $\mathbf{X}^\f \prec \widetilde{\mathbf{Y}}$. For completeness, we provide a proof of this result.

\begin{proof} [Proof of $\mathbf{X}^\f \prec \widetilde{\mathbf{Y}}$]
Let the current state of our coupling be 
$(\mathbf{x}, \widetilde{\mathbf{y}}) \in \N^2 \times \N^2$ with 
$\widetilde{\mathbf{y}} \geq \mathbf{x}$. We see $\mathbf{x}$ as the 
``small'' system and we index its customers by $(i,k)$ with 
$i \in \{1, 2\}$ (the user class) and $k = 1, \ldots, x_i$. The ``big'' system $\widetilde{\mathbf{y}}$ has the same customers and also additional ones which we label $(i,-k)$ with 
$i \in \{1,2\}$ and $k = 1, \ldots, \tilde y_i - x_i$. The next transition is built as follows:
		\begin{itemize}
			\item at rate $\lambda_1$, go to 
			$(\mathbf{x}+\mathbf{e}_1, \widetilde{\mathbf{y}}+\mathbf{e}_1)$;
			\item at rate $\lambda^\f_\tot$, go to 
			$(\mathbf{x}+\mathbf{e}_2, \widetilde{\mathbf{y}}+\mathbf{e}_2)$;
			\item each customer $(2,k)$ of type $2$ has an exponential clock with parameter $\theta$ and leaves the system if it rings: note that if $k < 0$ this only affects the big system, while if $k > 0$ this affects both systems;
			\item at rate $\mu$, do the following:
			\begin{enumerate}
				\item choose a customer $\widetilde C$ from the big system uniformly at random, i.e.,
$$
\P(\widetilde C = (i,k)) = \frac{1}{\tilde y_1 + \tilde y_2};
$$
				\item if $\widetilde C$ is in the small system, let 
				$C = \widetilde C$;
				\item else, let $C$ be chosen uniformly at random in the small system independently from everything else;
			\end{enumerate}
			Then remove the customer $C$ from the small system, and remove the customer $\widetilde C$ from the big system if it is of type $1$.
		\end{itemize}
		This construction is such that
		\begin{itemize}
			\item if a class $i$ customer arrives in the small system it also arrives in the big system;
			\item if a class $i$ customer leaves the big system and not the small one, then this customer was an ``additional'' customer which was in the big system but not in the small one.
		\end{itemize}
In particular, this construction leads to a state 
$(\mathbf{x'}, \widetilde{\mathbf{y'}})$ with 
$\widetilde{\mathbf{y'}} \geq \mathbf{x'}$. Moreover, the small system has the same dynamics as $\mathbf{X}^\f$ because $C$ is chosen uniformly at random in the small system, and the big system has the same dynamic as 
$\widetilde{\mathbf{Y}}$. Thus, this indeed builds a coupling of 
$\mathbf{X}^\f$ and $\widetilde{\mathbf{Y}}$ with $\mathbf{X}^\f \leq \widetilde{\mathbf{Y}}$, as desired.
\end{proof}

\subsection{Second coupling: $\widetilde{\mathbf{Y}} \prec \mathbf{Y}'$}
Starting from $\widetilde{\mathbf{Y}}$, we build $\mathbf{Y}'$ by lowering the service rate of $\widetilde Y_1$: when $\widetilde Y_2$ is larger than some threshold $\ell$, we put the service to~$0$, and when $\widetilde Y_2 \leq \ell$, we put $\mu y_1 / (y_1 + \ell)$ instead of $\mu y_1 / (y_1 + y_2)$, the former being indeed smaller than the latter when $y_2 \leq \ell$. More precisely, we fix $\varepsilon > 0$ (which is omitted from the notation for convenience) and we define $\ell = (1+\varepsilon) \lambda^\f_\tot / \theta$ and $\mathbf{Y}'$ the $\N^2$-valued Markov process with 
non-zero transition rates
\[
\mathbf{y} \in \N^2 \longrightarrow 
\begin{cases}
	\mathbf{y}+\mathbf{e}_1 & \text{ at rate } \lambda_1,
	\\ \\
	\mathbf{y}+\mathbf{e}_2 & \text{ at rate } \lambda^\f_\tot,
	\\ \\
	\mathbf{y}-\mathbf{e}_1 & \text{ at rate } 
	\displaystyle{\frac{\mu \, y_1}{y_1 + \ell}} \cdot 
	\Indicator{y_2 \leq \ell},
	\\ \\
	\mathbf{y}-\mathbf{e}_2 & \text{ at rate } \theta y_2,
\end{cases} 
\]
so that $\widetilde{\mathbf{Y}} \prec \mathbf{Y}'$ (in contrast to the inequality $\mathbf{X}^\f \prec \widetilde{\mathbf{Y}}$, the proof bears no difficulty and is thus omitted). Since 
$\mathbf{X}^\f \prec \widetilde{\mathbf{Y}}$, this gives 
$\mathbf{X}^\f \prec \mathbf{Y}'$ as desired.

Note that $Y'_2$ is an $M/M/\infty$ queue, so that $Y'_2(\infty)$ follows a Poisson distribution with parameter $\lambda^\f_\tot / \theta$. In particular, we obtain the convergence $Y'_2(\infty) / \lambda^\f_\tot \Rightarrow_{\lambda^\f_\tot} \theta^{-1} \pmb{\xi}^*_2$ and so we only have to prove that $Y'_1(\infty) / \lambda^\f_\tot \Rightarrow_{\lambda^\f_\tot, \varepsilon} \theta^{-1} \pmb{\xi}^*_1$ in order to prove Lemma~\ref{lemma:HT-f}. To do so we resort to another coupling and compare $Y'_1$ to a birth-and-death process $Y_1$.

\subsection{Third coupling: $\mathbf{Y} \prec \mathbf{Y}'$}
As $\ell$ is larger than the equilibrium point $\lambda^\f_\tot / \theta$ of $Y'_2$, excursions of $Y'_2$ above level $\ell$ are rare and so $Y'_1$ is only rarely turned off. For this reason, it is natural to compare 
$\mathbf{Y}'$ with the process obtained by putting the indicator function 
$\Indicator{y_2 \leq \ell}$ to $1$. To do so, let
$$
S = \{ (y_1, y'_1, y_2) \in \N^3: y'_1 \geq y_1\} \subset \N^3;
$$
we directly build the coupling that we need and consider $(Y_1, Y'_1, Y_2)$ the $S$-valued Markov process with the following non-zero transition rates:
\begin{equation} \label{eq:Y-Y'}
(y_1, y'_1, y_2) \in S \longrightarrow \begin{cases}
	(y_1, y'_1, y_2+1) & \text{ at rate } \lambda^\f_\tot,\\
	(y_1, y'_1, y_2-1) & \text{ at rate } \theta y_2,\\
	(y_1+1, y'_1+1, y_2) & \text{ at rate } \lambda_1,\\
	(y_1-1, y'_1-1, y_2) & \text{ at rate } \mu \alpha(y_1) \Indicator{y_2 \leq \ell},\\
	(y_1-1, y'_1, y_2) & \text{ at rate } \mu \alpha(y_1) \Indicator{y_2 > \ell},\\
	(y_1, y'_1-1, y_2) & \text{ at rate } \mu \beta_{y'_1 - y_1}(y_1) \Indicator{y_2 \leq \ell}
\end{cases} 
\end{equation}
with
\[ \alpha(y) = \frac{y}{y + \ell} \ \text{ and } \ beta_\delta(y) = \alpha(y + \delta) - \alpha(y) = 
\frac{\ell \delta}{(y + \ell) (y + \ell + \delta)}. 
\]
In words, what this process does is the following:
\begin{itemize}
	\item $Y_1$ and $Y_2$ are independent Markov processes;
	\item $Y_1$ is a state-dependent single-server queue with arrival rate $\lambda_1$ and instantaneous service rate $\mu \alpha(y_1)$ when in state $y_1$;
	\item $Y_2$ is an $M/M/\infty$ queue with arrival rate $\lambda^\f_\tot$ and service rate $\theta$;
	\item $Y'_1$ has the same arrivals than $Y_1$, but departures are different: there are additional departures at rate $\beta_{y'_1 - y_1}(y_1)$ when $Y_2 \leq \ell$, and no departure when $Y_2 > \ell$.
\end{itemize}
Since the function $\beta$ has been chosen so that
\[ \beta_{y'_1 - y_1}(y_1) + \alpha(y_1) = \alpha(y'_1) \]
we see that $(Y'_1, Y_2)$ is a Markov process with the same transition matrix than $\mathbf{Y}'$, and so we will actually write 
$\mathbf{Y}' = (Y'_1, Y_2)$ and we have 
$\mathbf{Y}' \geq \mathbf{Y} := (Y_1, Y_2)$.

In particular, this coupling defines several Markov processes, such as 
$Y_1$, $Y_2$, $Y := (Y_1, Y_2)$, $Y' = (Y'_1, Y_2)$ and $(Y_1, Y'_1, Y_2)$. For ease of notation, we will use the notation $\E_{\mathbf{x}}$ to denote the law of these Markov processes starting at $\mathbf{x}$, where the dimension of $\mathbf{x}$ depends on the process considered. For instance, if $\sigma$ is measurable with respect to $Y_2$ and $\varphi: \N \to \R_+$ is measurable, we will use the notation
\[ 
\E_\ell(\sigma), \ \E_{z, \ell} \left( \int_0^\sigma \varphi \circ Y_1 \right), \ \E_{z, \ell} \left( \int_0^\sigma \varphi \circ Y'_1 \right) \ \text{ or } \ \E_{z, z', \ell} \left( \int_0^\sigma \left( \varphi \circ Y_1 - \varphi \circ Y'_1 \right) \right) 
\]
that actually means
\[ \E_\ell(\sigma) = \E(\sigma \mid Y_2(0) = \ell), \]
\[ \E_{z, \ell} \left( \int_0^\sigma \varphi \circ Y_1 \right) = \E \left( \int_0^\sigma \varphi \circ Y_1 \mid Y_2(0) = \ell, Y_1(0) = z \right), \]
\[ \E_{z, \ell} \left( \int_0^\sigma \varphi \circ Y'_1 \right) = \E \left( \int_0^\sigma \varphi \circ Y'_1 \mid Y_2(0) = \ell, Y'_1(0) = z \right) \]
and
\begin{multline*}
	\E_{z, z', \ell} \left( \int_0^\sigma \left( \varphi \circ Y_1 - \varphi \circ Y'_1 \right) \right)\\
	= \E \left( \int_0^\sigma \left( \varphi \circ Y_1 - \varphi \circ Y'_1 \right) \mid Y_1(0) = z, Y'_1(0) = z', Y_2(0) = \ell \right).
\end{multline*}

Recall that the goal is to prove that $Y'_1(\infty) / \lambda^\f_\tot \Rightarrow_{\lambda^\f_\tot, \varepsilon} \theta^{-1} \pmb{\xi}^*_1$: what we will do is first prove this result for $Y_1$, which is much simpler since $Y_1$ is a birth-and-death process (whereas $Y'_1$ on its own is not Markov) and then transfer this result to $Y'_1$.

\subsection{Control of $Y_1$} 
\label{sub:control-Y1}
Let us now prove that $Y_1(\infty) / \lambda^\f_\tot \Rightarrow_{\lambda^\f_\tot, \varepsilon} \theta^{-1} \pmb{\xi}^*_1$. Let first 
$y^\pm = (1 \pm \varepsilon) \ell \pmb{\xi}^*_1$. Since the function 
$\alpha$ is increasing, when $Y_1$ is above level $y^+$ its departure rate is at least 
$\mu\alpha(y^+)$. Thus, if $L^+$ is an $M/M/1$ queue with arrival rate $\lambda_1$ and departure rate $\mu\alpha(y^+)$, we have 
$Y_1 \prec L^+ + y^+$. Likewise, if $L^-$ is an $M/M/1$ queue with arrival rate $\mu\alpha(y^-)$ and departure rate $\lambda_1$, we have 
$y^- - L^- \prec Y_1$.

Recall that $\pmb{\xi}^*_1 = \varrho_1 / (1-\varrho_1)$ and that $\varrho_1 < 1$: the load of $L^+$ is
\[ 
\frac{\lambda_1}{\mu \alpha(y^+)} = 
\frac{\varrho_1 ((1 + \varepsilon) \ell \pmb{\xi}^*_1 + \ell)}{(1 + \varepsilon) \ell \pmb{\xi}^*_1} = 
\frac{1 + \varepsilon \varrho_1}{1 + \varepsilon} < 1 
\]
and the load of $L^-$ is
\[ 
\frac{\mu \alpha(y^-)}{\lambda_1} = 
\frac{1 - \varepsilon}{1 - \varepsilon \varrho_1} < 1. 
\]
We thus deduce that $L^\pm$ are subcritical $M/M/1$ queues (uniformly in $\lambda^\f_\tot$, with $\varepsilon > 0$ fixed), so that
\[ 
\frac{L^\pm(\infty)}{\lambda^\f_\tot} \Rightarrow_{\lambda^\f_\tot} 0. 
\]
Since $y^\pm / \lambda^\f_\tot \to_{\lambda^\f_\tot} (1 \pm \varepsilon) (1 + \varepsilon) \theta^{-1} \pmb{\xi}^*_1$, we obtain
\[ 
\frac{y^- - L^-(\infty)}{\lambda^\f_\tot}, \quad 
\frac{L^+(\infty) - y^+}{\lambda^\f_\tot} \Rightarrow_{\lambda^\f_\tot, \varepsilon} \theta^{-1} 
\pmb{\xi}^*_1 
\]
and in view of $y^- - L^- \prec Y_1 \prec L^+ - y^+$, we finally get the desired result for $Y_1(\infty)$, namely 
$Y_1(\infty) / \lambda^\f_\tot \Rightarrow_{\lambda^\f_\tot, \varepsilon} \theta^{-1} \pmb{\xi}^*_1$.

\subsection{Transfer to $Y'_1$} 
\label{sub:transfer}
We now transfer the result for $Y_1(\infty)$ to $Y'_1(\infty)$ thanks to their coupling~\eqref{eq:Y-Y'}. Recall that $Y_1$ and $Y'_1$ obey the same dynamics, with the exception that service in $Y'_1$ is interrupted when $Y_2$ makes excursions above $\ell$. To compare their stationary distributions, we consider their trajectories over cycles of $Y_2$, where a cycle starts when $Y_2 = \ell$ and ends when $Y_2$ returns to $\ell$ from above: so there is a long period corresponding to $Y_2 \leq \ell$ where $Y_1$ and $Y'_1$ have the same dynamics, $Y'_1 \geq Y_1$ and they get closer (because the departure rate from $Y'_1$ is larger), and then a short period when $Y_2 \geq \ell+1$ where departures from $Y'_1$ are turned off and $Y'_1$ and $Y_1$ get further apart (when there is a departure from $Y_1$). Considering such cycles makes the comparison between $Y_1$ and $Y'_1$ tractable.

To formalize this idea, define recursively the stopping times 
$\sigma_0 = 0$ and 
\[ 
\tau_k = \inf \left\{ t \geq \sigma_k: Y_2(t) \geq \ell+1 \right\},
\qquad 
\sigma_{k+1} = \inf \left\{ t \geq \tau_k: Y_2(t) = \ell \right\} 
\]
and let $Z_k = Y_1(\sigma_k)$, $Z'_k = Y'_1(\sigma_k)$. Note that $Z$ and 
$Z'$ are ergodic Markov chains. Let $Z_\infty$ and $Z'_\infty$ be their respective stationary distribution and note, since $Y_2$ and $Y_1$ are independent, that $Z_\infty = Y_1(\infty)$ in distribution.

Let now $\sigma = \sigma_1$ and $\tau = \tau_0$; for any function 
$\varphi: \N \to \R_+$, define the functions $\Psi_\varphi$ and 
$\Psi'_\varphi$ by
\[ \Psi_\varphi(z) = \E_{z, \ell} \left( \int_0^\sigma \varphi \circ Y_1 \right) \ \text{ and } \ \Psi'_\varphi(z) = \E_{z, \ell} \left( \int_0^\sigma \varphi \circ Y'_1 \right). \]
The following result then relates the stationary distribution of $Y_1(\infty)$ and $Y'_1(\infty)$ to that of $Z_\infty$ and $Z'_\infty$, respectively. In the sequel, we write $\lVert f \rVert = \sup_{t \geq 0} \lvert f(t) \rvert$ for the $L_\infty$-norm of a function $f: \R_+ \to \R$.

\begin{lemma}
	For any bounded function $\varphi: \N \to \R_+$ we have
	\[ \E \left[ \varphi(Y_1(\infty)) \right] = \frac{1}{\E_{\ell}(\sigma)} \E \left[ \Psi_\varphi(Z_\infty) \right] \ \text{ and } \ \E \left[ \varphi(Y'_1(\infty)) \right] = \frac{1}{\E_{\ell}(\sigma)} \E \left[ \Psi_\varphi(Z'_\infty) \right]. 
	\]
	\label{Lemma3.1}
\end{lemma}

\begin{proof}
	We present the arguments only for $Y'_1$, as the same arguments apply to 
	$Y_1$. In this proof, $\to$ denotes the almost sure convergence as 
	$n \to \infty$. Since $(\sigma_k)_{k \geq 0}$ is a (possibly delayed) renewal process, by the strong Markov property, we have
	\[ 
	\frac{1}{n} \int_0^{\sigma_n} \varphi \circ Y'_1 \to \E_{\ell}(\sigma) \times \E(\varphi(Y'_1(\infty))). 
	\]
The rest of the proof is devoted to showing that we also have
	\[ 
	\frac{1}{n} \int_0^{\sigma_n} \varphi \circ Y'_1 \to \E(\Psi_\varphi(Z'_\infty)). 
	\]
	
Recall that $[\sigma_k, \sigma_{k+1}]$ represents the $k$th cycle of 
$Y_2$: $Y_2(\sigma_k) = \ell$, then $Y_2$ reaches $\ell+1$ at time 
$\tau_k$ and goes back to $\ell$ at time $\sigma_{k+1}$. For each cycle, 
$Y'_1$ starts in a random location $Y'_1(\sigma_k)$: call $i$-th $z$-cycle the $i$-th cycle of $Y_2$ such that $Y'_1$ starts in $z$, and denote its corresponding time interval by $[\sigma_i(z), \sigma_{i+1}(z)]$. If
	\[ 
	\Upsilon_i(z) = \int_{\sigma_i(z)}^{\sigma_{i+1}(z)} \varphi \circ Y'_1 
	\]
	represents the ``reward'' accumulated along the $i$-th $z$-cycle, then writing
	\[ 
	N_n(z) = \sum_{k=0}^{n-1} \Indicator{Z'_k = z} 
	\]
for the number of $z$-cycles starting before time $n$, partitioning the cycles depending on their starting point (for $Y'_1$) provides
	\[ 
	\frac{1}{n} \int_0^{\sigma_n} \varphi \circ Y'_1 = \frac{1}{n} \sum_{z \geq 0} \sum_{k=1}^{N_n(z)} \Upsilon_i(z) = \sum_{z \geq 0} \frac{N_n(z)}{n} \times \frac{1}{N_n(z)} \sum_{i=1}^{N_n(z)} \Upsilon_i(z). 
	\]
The ergodic theorem for $Z'$ implies that $N_n(z)/n \to 
\P(Z'_\infty = z)$. Moreover, for each fixed $z$, the 
$(\Upsilon_k(z), k \geq 0)$ are i.i.d.~with common distribution that of $\int_0^\sigma \varphi \circ Y'_1$ under $\P_{z, \ell}$, so the strong law of large numbers implies that
	\[ \frac{1}{N_n(z)} \sum_{i=1}^{N_n(z)} \Upsilon_i(z) \to \E_{z, \ell} \left( \int_0^\sigma \varphi \circ Y'_1 \right) = \Psi_\varphi(z). \]
Wrapping up, this suggests that
	\[ 
	\sum_{z \geq 0} \frac{N_n(z)}{n} \times \frac{1}{N_n(z)} \sum_{i=1}^{N_n(z)} \Upsilon_i(z) \to \sum_{z \geq 0} \P(Z'_\infty = n) \Psi_\varphi(z) 
	\]
	which is equal to $\E(\Psi_\varphi(Z'_\infty))$, as desired. Let us justify the latter assertion. If we restrict the sum in $z$ to a finite number of terms, then the previous arguments can then be applied and they give, for any $z^* \geq 0$,
	\[ 
	\frac{1}{n} \sum_{z \leq z^*} \sum_{i=1}^{N_n(z)} \Upsilon_i(z) \to \sum_{z \leq z^*} \P(Z'_\infty = z) \Psi_\varphi(z) = \E(\Psi_\varphi(Z'_\infty) ; Z'_\infty \leq z^*). 
	\]
	Since $\Psi_\varphi \geq 0$ (because $\varphi \geq 0$), monotone convergence implies that the above right-hand side converges to $\E(\Psi_\varphi(Z'_\infty))$ as $z^* \to \infty$. Thus, in order to complete the proof, it remains to show that
	\[ \limsup_{n \to \infty} \frac{1}{n} \sum_{z \geq z^*} \sum_{i=1}^{N_n(z)} \Upsilon_i(z) \xrightarrow[z^* \to \infty]{} 0. \]
	We have
	\[ \frac{1}{n} \sum_{z \geq z^*} \sum_{i=1}^{N_n(z)} \Upsilon_i(z) \leq \frac{\lVert \varphi \rVert}{n} \sum_{z \geq z^*} \sum_{i \leq N_n(z)} \left( \sigma_{i+1}(z) - \sigma_i(z) \right) = \frac{\lVert \varphi \rVert}{n} \sum_{k=0}^{n-1} \delta_{k+1} \Indicator{Z_k \geq z^*} \]
	where $\delta_{k+1} = \sigma_{k+1} - \sigma_k$. The process $(\delta_{k+1}, Z'_k)$ is Markov: given the past until time $k$, $\delta_{k+2}$ is independent and distributed according to $\sigma$ under $\P_\ell$, and $Z'_{k+1}$ corresponds to the evolution of $Y'_1$ in-between a $Z'_k$-cycle of $Y_2$ with length $\delta_{k+1}$. Note that, because each sequence 
	$(\delta_{k+1})$ and $(Z'_k)$ is tight, the sequence 
	$(\delta_{k+1}, Z_k)$ is also tight and since it is also Markov, it converges to $(\delta_\infty, Z'_\infty)$, say. Thus, the ergodic theorem implies that
	\[ \frac{1}{n} \sum_{k=0}^{n-1} \delta_{k+1} \Indicator{Z_k \geq z^*} \to \E \left( \delta_\infty ; Z'_\infty \geq z^* \right) \]
	and since $\E(\delta_\infty) = \E_\ell(\sigma) < \infty$, we obtain the desired result by monotone convergence and letting $z^* \to \infty$.
\end{proof}

By Lemma \ref{Lemma3.1}, we thus deduce that
\begin{equation} \label{eq:relation}
	\E \left[ \varphi(Y(\infty)) \right] - \E \left[ \varphi(Y'(\infty)) \right] = \frac{1}{\E_{\ell}(\sigma)} \left( \E \left[ \Psi_\varphi(Z_\infty) \right] - \E \left[ \Psi'_\varphi(Z'_\infty) \right] \right).
\end{equation}
To control the right-hand side, we decompose it as
\[ \big( \E \left[ \Psi_\varphi(Z_\infty) \right] - \E \left[ \Psi'_\varphi(Z_\infty) \right] \big) + \big( \E \left[ \Psi'_\varphi(Z_\infty) \right] - \E \left[ \Psi'_\varphi(Z'_\infty) \right] \big) \]
and control each difference in the next two lemmas.

\begin{lemma}
	We have
	\[ \left \lvert \E \left[ \Psi_\varphi(Z_\infty) \right] - \E \left[ \Psi'_\varphi(Z_\infty) \right] \right \rvert \leq 2 \lVert \varphi \rVert \E_{\ell+1}(\sigma). \]
\end{lemma}

\begin{proof}
	Using the strong Markov property at time $\tau$, we can write
	\[ \Psi_\varphi(z) = \E_{z, \ell} \left( \int_0^\tau \varphi \circ Y_1 \right) + \sum_{z' \geq 0} \P_{z,\ell} \left( Y_1(\tau) = z' \right) \E_{z', \ell+1} \left( \int_0^\sigma \varphi \circ Y_1 \right) \]
	and
	\begin{align*}
		\Psi'_\varphi(z) & = \E_{z, \ell} \left( \int_0^\tau \varphi \circ Y'_1 \right) + \sum_{z' \geq 0} \P_{z,\ell} \left( Y'_1(\tau) = z' \right) \E_{z', \ell+1} \left( \int_0^\sigma \varphi \circ Y'_1 \right)\\
		& = \E_{z, \ell} \left( \int_0^\tau \varphi \circ Y_1 \right) + \sum_{z' \geq 0} \P_{z,\ell} \left( Y_1(\tau) = z' \right) \E_{z', \ell+1} \left( \int_0^\sigma \varphi \circ Y'_1 \right)
	\end{align*}
	where the second equality comes from the fact that $Y_1$ and $Y'_1$ coincide on $[0,\tau]$ if they start at the same level. Since $Y_1$ is independent from $Y_2$ (and hence from $\tau$) and $Y$ and $Z$ have the same stationary distribution, we have
	\begin{multline*}
		\E \left[ \Psi_\varphi(Z_\infty) \right] - \E \left[ \Psi'_\varphi(Z_\infty) \right] = \sum_{z \geq 0} \P(Z_\infty = z) \E_{z, z, \ell+1} \left( \int_0^\sigma \varphi \circ Y_1 \right)\\
		- \sum_{z \geq 0} \P(Z_\infty = z) \E_{z, z, \ell+1} \left( \int_0^\sigma \varphi \circ Y'_1 \right)
	\end{multline*}
	from which the result follows.
\end{proof}

\begin{lemma}
	If $\varphi: \R_+ \to \R$ is differentiable with derivative $\varphi'$, then we have
	\[ \left \lvert \E \left[ \Psi'_\varphi(Z'_\infty) \right] - \E \left[ \Psi'_\varphi(Z_\infty) \right] \right \rvert \leq \lVert \varphi' \rVert \E(Z'_\infty - Z_\infty) \E_\ell(\sigma). \]
\end{lemma}

\begin{proof}
	Fix temporarily $z' \geq z$. Using the relation
	\[ \E_{z, \ell} \left( \int_0^\sigma \varphi \circ Y'_1 \right) = \E_{z, \ell} \left( \int_0^\tau \varphi \circ Y_1 \right) + \E_{z, \ell} \left( \int_\tau^\sigma \varphi \circ Y'_1 \right) \]
	owing to the fact that $Y_1$ and $Y'_1$ have the same dynamics on 
	$[0,\tau]$, write
	\begin{multline*}
		\Psi'_\varphi(z') - \Psi'_\varphi(z) = \E_{z', \ell} \left( \int_0^\tau \varphi \circ Y_1 \right) - \E_{z, \ell} \left( \int_0^\tau \varphi \circ Y_1 \right)\\
		+ \E_{z', \ell} \left( \int_\tau^\sigma \varphi \circ Y'_1 \right) - \E_{z, \ell} \left( \int_\tau^\sigma \varphi \circ Y'_1 \right).
	\end{multline*}
To compute
	\[ 
	\E_{z', \ell} \left( \int_0^\tau \varphi \circ Y_1 \right) - \E_{z, \ell} \left( \int_0^\tau \varphi \circ Y_1 \right), 
	\]
	we couple $Y_1$ starting from two different initial conditions $z$ and 
	$z'$: in fact, when considered on $[0, \tau]$, this is exactly what the coupling between $Y_1$ and $Y'_1$ does, and so we thus have
	\[ 
	\E_{z', \ell} \left( \int_0^\tau \varphi \circ Y_1 \right) - \E_{z, \ell} \left( \int_0^\tau \varphi \circ Y_1 \right) = \E_{z, z', \ell} \left( \int_0^\tau \left( \varphi \circ Y'_1 - \varphi \circ Y_1 \right) \right) 
	\]
and so
	\begin{align*}
		\left \lvert \E_{z', \ell} \left( \int_0^\tau \varphi \circ Y_1 \right) - \E_{z, \ell} \left( \int_0^\tau \varphi \circ Y_1 \right) \right \rvert & \leq \lVert \varphi' \rVert \E_{z, z', \ell} \left( \int_0^\tau \left( Y'_1 - Y_1 \right) \right)\\
		& \leq \lVert \varphi' \rVert (z' - z) \E_\ell (\tau)
	\end{align*}
with the last inequality coming from the fact that $Y'_1 - Y_1$ is non-increasing on $[0, \tau]$.
	
	We now control the difference
	\[ 
	\E_{z', \ell} \left( \int_\tau^\sigma \varphi \circ Y'_1 \right) - 
	\E_{z, \ell} \left( \int_\tau^\sigma \varphi \circ Y'_1 \right). 
	\]
	Consider $(\Upsilon, \Upsilon')$ and $A$ such that
	\begin{itemize}
		\item $(\Upsilon, \Upsilon')$, $A$ and $Y_2$ are mutually independent;
		\item $(\Upsilon, \Upsilon')$ is distributed as $(Y_1(\tau), Y'_1(\tau))$ under $\P_{z, z', \ell}$;
		\item $A$ is a Poisson process with intensity $\lambda_1$.
	\end{itemize}
On interval $[\tau, \sigma]$, $Y'_1 - Y'_1(\tau)$ is simply a Poisson process distributed as $A$: the strong Markov property therefore implies that
	\begin{multline*}
		\E_{z', \ell} \left( \int_\tau^\sigma \varphi \circ Y'_1 \right) - \E_{z, \ell} \left( \int_\tau^\sigma \varphi \circ Y'_1 \right)\\
		= \E_{\ell+1} \left( \int_0^\sigma \left( \varphi(\Upsilon' + A(s)) - \varphi(\Upsilon + A(s)) \right) \d s \right)
	\end{multline*}
	and so
	\[ \left \lvert \E_{z', \ell} \left( \int_\tau^\sigma \varphi \circ Y'_1 \right) - \E_{z, \ell} \left( \int_\tau^\sigma \varphi \circ Y'_1 \right) \right \rvert \leq \lVert \varphi' \rVert \E_{z, z', \ell} \left( Y'_1(\tau) - Y_1(\tau) \right) \E_{\ell+1}(\sigma). \]
	Finally, using $Y'_1(\tau) \leq Y'_1(\sigma)$ and averaging over $(Z_\infty, Z'_\infty)$, we obtain
	\[ \left \lvert \E \left[ \Psi'_\varphi(Z'_\infty) \right] - \E \left[ \Psi'_\varphi(Z_\infty) \right] \right \rvert \leq \lVert \varphi' \rVert \E(Z'_\infty - Z_\infty) (\E_\ell(\tau) + \E_{\ell+1}(\sigma)) \]
	from which the result follows since $\E_\ell(\sigma) = \E_\ell(\tau) + \E_{\ell+1}(\sigma)$ by the strong Markov property.
\end{proof}

Plugging in the bounds of the two previous lemmas 
into~\eqref{eq:relation}, we conclude that for any function 
$f: \R_+ \to \R_+$ bounded, differentiable and with bounded derivative, we have
\begin{multline*}
	\left \lvert \E \left[ f \left( \frac{Y_1(\infty)}{\lambda^\f_\tot} \right) \right] - \E \left[ f \left( \frac{Y'_1(\infty)}{\lambda^\f_\tot} \right) \right] \right \rvert\\
	\leq \frac{1}{\E_\ell(\sigma)} \left( 2 \lVert f \rVert \E_{\ell+1}(\sigma) + \frac{1}{\lambda^\f_\tot} \lVert f' \rVert \E(Z'_\infty - Z_\infty) \E_\ell(\sigma) \right)
\end{multline*}
hence
\[ \left \lvert \E \left[ f \left( \frac{Y_1(\infty)}{\lambda^\f_\tot} \right) \right] - \E \left[ f \left( \frac{Y'_1(\infty)}{\lambda^\f_\tot} \right) \right] \right \rvert \leq 2 \lVert f \rVert \frac{\E_{\ell+1}(\sigma)}{\E_\ell(\sigma)} + \frac{1}{\lambda^\f_\tot} \lVert f' \rVert \E(Z'_\infty - Z_\infty). \]

The two following lemmas therefore imply that
\[ \E \left[ f \left( \frac{Y_1(\infty)}{\lambda^\f_\tot} \right) \right] - \E \left[ f \left( \frac{Y'_1(\infty)}{\lambda^\f_\tot} \right) \right] \to 0 \]
as $\lambda^\f_\tot \to \infty$ for any differentiable, bounded function 
$f$ with bounded derivative. Since $Y_1(\infty) / \lambda^\f_\tot \Rightarrow_{\lambda^\f_\tot, \varepsilon} \theta^{-1} \pmb{\xi}^*_1$, this implies that
\[ \E \left[ f \left( \frac{Y'_1(\infty)}{\lambda^\f_\tot} \right) \right] \to_{\lambda^\f_\tot, \varepsilon} f(\theta^{-1} \pmb{\xi}^*_1) \]
which shows that $Y'_1(\infty) / \lambda^\f_\tot \Rightarrow_{\lambda^\f_\tot, \varepsilon} \theta^{-1} \pmb{\xi}^*_1$, as claimed.

\begin{lemma}
As $\lambda^\f_\tot \to \infty$, we have 
$\E_{\ell+1}(\sigma) / \E_\ell(\sigma) \to 0$.
\end{lemma}

\begin{proof}
For $\gamma > 0$, let $T^\gamma$ be the hitting time of $0$ by an $M/M/1$ queue started at $1$ and with input rate $\gamma$ and output rate 
$(1+\varepsilon) \gamma$. Above level $\ell+1$, $Y_2$ is upper bounded by an $M/M/1$ queue with input rate $\lambda^\f_\tot$ and output rate 
$\theta \ell = (1+\varepsilon) \lambda^\f_\tot$, so that 
$\sigma \prec T^{\lambda^\f_\tot}$, where $\sigma$ is considered under 
$\P_{\ell+1}$. Since $T^\gamma = T^1/\gamma$ in distribution, this yields 
$\E_{\ell+1}(\sigma) \leq \E(T^1) / \lambda^\f_\tot$. Since clearly 
$\E_\ell(\sigma) \to \infty$, we obtain the result.
\end{proof}

\begin{lemma}
As $\lambda^\f_\tot \to \infty$, we have 
$\limsup \E(Z'_\infty - Z_\infty) < \infty$. In particular, 
$\E(Z'_\infty - Z_\infty) / \lambda^\f_\tot \to 0$.
\end{lemma}

\begin{proof}
Let $\Delta_k = Z'_k - Z_k$. The idea is that when $\Delta_k$ is large, then on $[\sigma_k,\tau_k]$ the function $\beta_{Y_1 - Y'_1}$ takes (relatively) large values which brings the processes $Y'_1$ and $Y_1$ closer and makes $\Delta_{k+1}$ smaller. To formalize this idea, we use Theorem~$2.3$ in~\cite{Hajek82:0}: to apply this result, we need to control the exponential moments of $\Delta_1$. To do so, we consider $P$ and $P'$ two Poisson point processes on $\R_+ \times [0,1]$ with intensity $\mu \d t \otimes \d x$ so that $P$, $P'$, $Y_1$ and $Y_2$ are independent, and we write
	\begin{multline} 
	\label{eq:Delta}
		\Delta_1 - \Delta_0 = - \int \Indicator{0 \leq s \leq \tau, \zeta \leq \beta_{Y'_1(s-) - Y_1(s-)}(Y_1(s-))} P'(\d s \d \zeta)
		\\
		+ \int \Indicator{\tau \leq s \leq \sigma, \zeta \leq \alpha(Y'_1(s-))} P(\d s \d \zeta).
	\end{multline}
	The first negative term translates the fact that on $[0,\tau]$, $Y'_1$ and $Y_1$ get closer at rate $\mu \beta_{Y'_1 - Y_1}(Y_1)$ (which is the rate at which there is a departure from $Y'_1$ and not from $Y_1$), while on $[\tau, \sigma]$ they get further apart at rate $\mu \alpha(Y_1)$ 
(which is the rate at which there is a departure from $Y_1$). Moreover, as in the previous proof, let $T^\gamma$ be the hitting time of $0$ by an 
$M/M/1$ queue started at $1$ and with input rate $\gamma$ and output rate 
$(1+\varepsilon) \gamma$, and $\eta$ be given by
	\[ 
	\mu (e^\eta - 1) = \lambda^\f_\tot \left( \sqrt{1+\varepsilon} - 1 \right)^2. 
	\]
	
	\noindent \textit{First case: $\Delta_0 = 0$.} If $\Delta_0 = 0$, we then have $Y_1 = Y'_1$ on $[0, \tau]$ and so~\eqref{eq:Delta} reduces to
	\[ 
	\Delta_1 - \Delta_0 = \int \Indicator{\tau \leq s \leq \sigma, \zeta \leq \alpha(Y'_1(s-))} P(\d s \d \zeta) \leq P^* := P([\tau, \sigma] \times [0,1]) 
	\]
	hence
	\[ 
	\E_{z, z, \ell} \left( e^{\eta \Delta_1} \right) \leq \E_{\ell+1} \left( e^{\eta P^*} \right) = \E_{\ell+1} \left( e^{\mu (e^\eta-1) \sigma} \right) \]
after using the strong Markov property for the first inequality and the fact that, under $\P_{\ell+1}$ and conditionally on $Y_2$, $P^*$ is a Poisson random variable with parameter $\mu \sigma$. Using the same argument as in the previous lemma, namely 
$\sigma \prec T^1 / \lambda^\f_\tot$, we obtain
	\[ 
	\E_{z, z, \ell} \left( e^{\eta \Delta_1} \right) \leq \E \left( e^{\mu (e^\eta-1) T^1 / \lambda^\f_\tot} \right) \leq c := (\varepsilon^{-1} - 1)^{1/2} 
	\]
where the last inequality is provided by Proposition~$5.4$ in~\cite{Robert03:0}. If $\E$ denotes the expectation under the stationary distribution of $(Z_k, Z'_k)$, the strong Markov property then entails
	\[ 
	\E \left( e^{\eta \Delta_{k+1}} \mid \cF_k \right) \Indicator{\Delta_k = 0} \leq c 
	\]
	where $\cF_k = \sigma((Y_1(s), Y'_1(s), Y_2(s)), s \leq \sigma_k)$.
	
	\noindent \textit{Second case: $\Delta_0 \geq 1$.} Next, consider the case $\Delta_0 \geq 1$. Then
	\[ 
	\int \Indicator{0 \leq s \leq \tau, \zeta \leq \beta_{Y'_1(s-) - Y_1(s-)}(Y_1(s-))} P'(\d s \d \zeta) 
	\]
counts the number of points of $P'$ that fall below the curve 
	$\beta_{Y'_1 - Y_1}(Y_1)$ before time~$\tau$. Each time a point falls below this curve, this makes $Y'_1 - Y_1$ decrease by one, and the $\beta$ curve lowers until $Y'_1 - Y_1$ possibly hits $0$ in which case $\beta_0 = 0$ and no more point can fall below this line. In particular, we have
	\[ \int \Indicator{0 \leq s \leq \tau, \zeta \leq \beta_{Y'_1(s-) - Y_1(s-)}(Y_1(s-))} P'(\d s \d \zeta) \geq B \]
	where $B$ is the Bernoulli random variable $B = \Indicator{I \geq 1}$ with
	\[ I = \int \Indicator{0 \leq s \leq \tau, \zeta \leq \beta_1(Y_1(s-))} P'(\d s \d \zeta). \]
	Indeed, if $B = 0$, then this inequality is true; if $B = 1$, then this means that $I \geq 1$, i.e., a point of $P'$ fell below the curve 
	$\beta_1(Y_1)$, and since $\beta_1 \leq \beta_{Y'_1(0) - Y_1(0)}$, this necessarily implies that a point of $P'$ fell below the curve 
	$\beta_{Y'_1 - Y_1}(Y_1)$, i.e., the left-hand side of the previous display is also $\geq 1$. We thus derive that 
	$\Delta_1 - \Delta_0 \leq P([\tau, \sigma] \times [0,1]) - B$ and by independence between $P$, $P'$, $Y_1$ and $Y_2$, the strong Markov property provides
	\[ 
	\E_{z, z', \ell} \left( e^{\eta(\Delta_1 - \Delta_0)} \right) \leq c \E_{z, \ell}(e^{-\eta B}). 
	\]
	Averaging with respect to $(Z_\infty, Z'_\infty)$ and using the strong Markov property, we obtain
	\[ \E \left( e^{\eta (\Delta_{k+1} - \Delta_k)} \mid \cF_k \right) \Indicator{\Delta_k \geq 1} \leq c \, \E_{\ell}(e^{-\eta B}) \]
	where here and in the rest of the proof, $\E_\ell$ corresponds to an initial state of $(Y_1, Y_2)$ distributed as $(Y_1(\infty), \ell)$. Assume at this stage that 
	\begin{equation} 
	\label{eq:claim}
		\E_{\ell}(e^{-\eta B}) \to 0
	\end{equation}
(we will prove this claim at the end of the proof). Then Theorem~$2.3$ 
in~\cite{Hajek82:0} gives, for large enough $\lambda^\f_\tot$,
	\[ 
	\P(\Delta_\infty \geq d) \leq 
	\frac{c \, e^\eta}{1 - c \, \E_{\ell}(e^{-\eta \, B})} e^{-\eta d} 
	\]
	from which we get
	\[ 
	\E(\Delta_\infty) = \sum_{d \geq 1} 
	\P(\Delta_\infty \geq d) \leq 
	\frac{c}{1 - c \, \E_{\ell}(e^{-\eta \, B})} \frac{1}{1-e^{-\eta}}. 
	\]
Since $e^\eta \to \infty$ with $\lambda^\f_\tot$, we obtain the desired result.
	
In order to conclude the proof, we now prove the claim~\eqref{eq:claim}. Since $B$ is a Bernoulli random variable, we have 
$\E_\ell(e^{-\eta B}) = e^{-\eta} \P(B = 1) + \P(B = 0)$ and as 
$\eta \to \infty$, we only have to show that $\P(B = 0) \to 0$. Let $Y^*_1 = \sup_{[0, \ell^2]} Y_1$. Since $\beta_1$ is decreasing, when $\tau \leq \ell^4$ and $Y^*_1 \leq \ell^2$, we have
	\begin{multline*}
		\int \Indicator{0 \leq s \leq \tau, \zeta \leq \beta_1(Y_1(s-))} P'(\d s \d \zeta)\\
		\leq I^* := \int \Indicator{0 \leq s \leq \ell^4, \zeta \leq \beta_1(\ell^2)} P'(\d s \d \zeta)
	\end{multline*}
	hence
	\begin{align*}
		\P(B = 0) & \leq \P(I^* = 0) + \P_\ell(\tau \leq \ell^4) + \P(Y^*_1 \geq \ell^2)\\
		& \leq \P(I^* = 0) + \P_\ell(\tau \leq \ell^4) + \P(Y^*_1 \geq \ell^2 \mid Y_1(0) \leq \ell^{3/2}) + \P(Y_1(\infty) \geq \ell^{3/2})
	\end{align*}
	where $Y_1(0)$ is distributed according to $Y_1(\infty)$. Since $I^*$ is a Poisson random variable with parameter $\mu \beta_1(\ell^2) \ell^4$, we have
	\[ \P(I^* = 0) = \exp \left( - \mu \ell^4 \beta_1(\ell^2) \right) \]
	which vanishes as $\lambda^\f_\tot \to \infty$ since $\beta_1(\ell^2)$ decays like $1/\ell^3$. Moreover, by proceeding as in 
	Section~\ref{sub:control-Y1} and comparing $Y_1$ with a subcritical 
	$M/M/1$, it is easy to see that $\P(Y_1(\infty) \geq \ell^{3/2}) \to 0$. It thus remains to control the two last terms $\P_\ell(\tau \leq \ell^4)$ and $\P(Y^*_1 \geq \ell^2 \mid Y_1(0) \leq \ell^{3/2})$, which can be done by comparison with a subcritical $M/M/1$ queue. In fact, it is well known that it takes an exponential time for a subcritical $M/M/1$ queue to reach high values (see for instance Proposition $5.11$ in~\cite{Robert03:0}) and we can compare $Y_1$ and $Y_2$ to such a queue to transfer this behavior to these two processes:
	
	\begin{itemize}
		\item for $Y_1$, we can use the fact that, when in the range $[\ell^{3/2}, \ell^2]$, it is smaller than a subcritical queue $M/M/1$ queue with input rate $\lambda_1$ and output rate $\mu \alpha(\ell^{3/2})$;
		\item for $Y_2$, we can use the fact that, when in the range 
		$[\ell', \ell]$ with lower bound 
		$\ell' = (\lambda^\f_\tot / \theta + \ell)/2$, it is smaller than a subcritical $M/M/1$ queue with input rate $\lambda^\f_\tot$ and output rate $\theta \ell'$.
	\end{itemize}
	The proof is thus complete.
\end{proof}


\section{Proof of Lemma~\ref{lemma:2}} 
\label{sec:proof-2}


Fix an integer $k \geq \mu / \theta$, let $S = \{-k, -k+1, \ldots, \}$ and $Z$ be the $S$-valued Markov process with non-zero transition rates
\[ z \in S \longrightarrow \begin{cases}
	z+1 & \text{ at rate } \lambda^\f_\tot,\\
	z-1 & \text{ at rate } \theta (k + z).
\end{cases} \]
Compared to the transition rates of $X^\f_2$, this amounts to upper bounding the rate $\mu x_2 / (x_1 + x_2)$ by $\theta k$ which makes $X_2$ smaller. Note also that the downward rate $\theta (k + z)$ is $0$ for 
$z = -k$, so $Z$ indeed lives in $S$. This implies $Z \prec X^\f_2$ and therefore
\[ 
\P(\mathbf{X}^\f(\infty) = \mathbf{0}) \leq \P(Z(\infty) = 0). 
\]
From its transition rates, it is apparent that $Z+k$ is an $M/M/\infty$ queue with input rate $\lambda^\f_\tot$ and service rate $\theta$, and so 
$Z(\infty) - k$ follows a Poisson distribution with parameter $\lambda^\f_\tot / \theta$. In particular,
\[ 
\P(Z(\infty) = 0) = e^{-\lambda^\f_\tot / \theta} \frac{(\lambda^\f_\tot / \theta)^k}{k!} 
\]
and so $-\log \P(Z(\infty) = 0) / \lambda^\f_\tot \to \theta^{-1}$ which gives the desired bound.


\section{Possible extensions} 
\label{sec:extensions}


In this paper, we aimed to prove the minimal result that shows that mobility makes delay increase like $-\log(1-\varrho)$ in heavy traffic, instead of the usual $1/(1-\varrho)$ scaling. However, we can go a bit further than Theorem~\ref{thm:main} by formulating an interesting conjecture, which we can only partially prove. In this section, we will also discuss the link with the Large Deviations theory, and possible extensions of our model.

\subsection{Extension of Theorem~\ref{thm:main}}
As explained earlier, Theorem~\ref{thm:main} is a direct consequence of Lemmas~\ref{lemma:HT-f} and~\ref{lemma:2}. With similar tools as those used in Section~\ref{sec:proof-HT-f}, it is actually possible to prove the following result which makes the result of Lemma~\ref{lemma:HT-f} more precise.

\begin{lemma}
Assume that $\theta> 0$ and $\varrho_1 < 1$. As 
$\lambda^\f_\tot \to \infty$, we then have
	\[ 
	\frac{1}{\lambda^\f_\tot} \mathbf{X}^\f(\infty) 
	\Rightarrow \theta^{-1} 
	\pmb{\xi}^*. 
	\]
In particular, as $\varrho \uparrow 1$, we have
	\[ 
	\frac{1}{\Lambda_\tot} \mathbf{X}(\infty) \Rightarrow 
	\theta^{-1} \pmb{\xi}^*. 
	\]
\end{lemma}

The idea to prove this result is to prove a matching lower bound to that already proved, by comparing $\mathbf{X}^\f$ to a lower bounding process 
$\mathbf{Y}'$ with non-zero transition rates
\[ 
\mathbf{y} \in \N^2 \longrightarrow \begin{cases}
	\mathbf{y}+\mathbf{e}_1 & \text{ at rate } \lambda_1,
	\\ \\
	\mathbf{y}+\mathbf{e}_2 & \text{ at rate } \lambda^\f_\tot,
	\\ \\
	\mathbf{y}-\mathbf{e}_1 & \text{ at rate } 
	\displaystyle{\mu \, \frac{y_1}{y_1 + \ell} \cdot \Indicator{y_2 \geq \ell} + \mu \cdot \Indicator{y_2 < \ell}},
	\\ \\
	\mathbf{y}-\mathbf{e}_2 & \text{ at rate } \mu + \theta y_2,
\end{cases} 
\]
We then have $\mathbf{Y}' \prec \mathbf{X}^\f$ and the analysis of 
$\mathbf{Y}'$ proceeds as in Sections~\ref{sub:control-Y1} and~\ref{sub:transfer} and leads to 
$\mathbf{Y}'(\infty) / \lambda^\f_\tot \Rightarrow_{\lambda^\f_\tot, \varepsilon} \theta^{-1} \pmb{\xi}^*$. We have here to choose $\ell = (1-\varepsilon) \lambda^\f_\tot / \theta$, so that excursions of $Y_2$ below level $\ell$ are rare, and thus as for the lower bound, $Y'_1$ essentially behaves as a birth-and-death process independent from $Y'_2$.
\\

What is much more difficult is to extend Lemma~\ref{lemma:2}. For various reasons, we believe that $X^\f_1(\infty)$ and $X^\f_2(\infty)$ are asymptotically independent and that
\begin{equation} \label{eq:0}
	\P(\mathbf{X}^\f(\infty) = \mathbf{0}) \approx 
	\P(X^\f_1(\infty) = 0) \times \P(X^\f_2(\infty) = 0),
\end{equation}
where the approximation is thought to hold in the logarithmic order. Actually, thanks to perturbation analysis, we know how to prove that
\[ 
\P(\mathbf{X}^\f(\infty) = \mathbf{0}) \geq \P(X^\f_1(\infty) = 0) \times 
\P(X^\f_2(\infty) = 0) 
\]
and we would need a matching upper bound. Moreover, we know how to control each probability in the latter right-hand side: for $X^\f_2(\infty)$, this is easily done via a comparison with $M/M/\infty$ queues. For 
$X^\f_1(\infty)$, this is more subtle but, as in 
Section~\ref{sec:proof-HT-f}, we can prove that $\P(X^\f_1(\infty) = 0)$ has the same exponential order than the corresponding birth-and-death process with death rate $\mu x / (x + \lambda^\f_\tot / \theta)$, obtained by replacing $x_2$ by its equilibrium value $\lambda^\f_\tot / \theta$. Thus, we can state the following result.

\begin{lemma} \label{lemma:infty}
	Assume that $\varrho_1 < 1$. As $\lambda^\f_\tot \to \infty$, we have
	\[ 
	- \frac{1}{\lambda^\f_\tot} \log \P(X^\f_1(\infty) = 0) \to - \theta \log(1-\varrho_1). 
	\]
\end{lemma}

In particular, these two results imply that any accumulation point of 
$\mathbf{X}^\f(\infty) / \lambda^\f_\tot$ is $> 0$, so that $-\log(1-\varrho)$ is indeed the right scale for $\mathbf{X}(\infty)$. Thus, as mentioned in Remark~\ref{rk:conj}, we actually believe that
\[ - \frac{1}{\lambda^\f_\tot} \log \P(\mathbf{X}^\f(\infty) = \mathbf{0}) \to - 
\theta \log(1-\varrho_1) - \theta \]
which would lead to the following conjecture.

\begin{conjecture} 
\label{conj}
	If $\theta > 0$, then
	\[ 
	\frac{\mathbf{X}(\infty)}{-\log(1-\varrho)} \Longrightarrow
	\frac{1}{1 - \log(1-\varrho_1)} \cdot \pmb{\xi}^*, 
	\]
as $\varrho \uparrow 1$, with point $\pmb{\xi}^*$ defined in Theorem 
\ref{thm:main}.
\end{conjecture}

If this statement were true, it would have the surprising feature that the heavy traffic limit is independent of the parameter $\theta$: all that matters is that $\theta > 0$, but the precise value is irrelevant in heavy traffic. Moreover, this would give the approximation
\[ 
X_1(\infty) + X_2(\infty) \approx M(\varrho_2) \cdot 
\log \left(\frac{1}{1-\varrho} \right) 
\]
for the total number of users with $\varrho_2 \approx 1 - \varrho_1$ and where $M(x) = 1/(x - x \log x)$. As the function 
$x \in [0,1] \mapsto x - x \log x$ is increasing, this would suggest that for a given load $\varrho$, the system performance is improved with a larger fraction $1-\varrho_1$ of mobile users.

\subsection{Large Deviations for processes undergoing 
time-scale separation} \label{sub:LD}
In order to prove the above conjecture, what we miss is formalizing the approximation~\eqref{eq:0}. There is a vast literature on Large Deviations for Markov processes; we did not find, however, any reference that fits our framework.

What is specific in our problem of controlling the stationary probability in $\mathbf{0} = (0,0)$ of $\mathbf{X}^\f$ is that the two components $X^\f_1$ and $X^\f_2$ evolve on different time-scales. When 
$\lambda^\f_\tot$ is large, Lemma~\ref{lemma:infty} shows that 
$\mathbf{X}^\f(\infty)$ is of the order of 
$\lambda^\f_\tot$. But $X^\f_1$ is similar to a birth-and-death process with bounded birth and death rates, which makes it evolve on the linear time scale proportional to $\lambda^\f_\tot$, while $X^\f_2$ is similar to an $M/M/\infty$ queue and thus evolves on a constant time-scale. The process $\mathbf{X}^\f$ therefore undergoes time-scale separation, or stochastic homogenization: when there are two components with different speeds, the stochastic homogenization principle asserts that the slow one 
(namely, $X^\f_1$) only interacts with the fast one (namely, $X^\f_2$) through its stationary distribution. Here the stationary distribution of 
$X^\f_2$ is essentially a Poisson random variable with parameter 
$\lambda^\f_\tot / \theta$ and is thus independent of $X^\f_1$, which leads to a simpler form of stochastic homogenization.

This stochastic averaging principle is well-known, and there is a rich literature on Large Deviations theory in this case, see for 
instance~\cite{Freidlin78:1, Huang16:0, Liptser96:0, Puhalskii16:0, Veretennikov99:0, Veretennikov13:0}. However, all these works only establish Large Deviations principles for the empirical measure of the fast process, which is admittedly the most natural question to address. What we presently need, however, is really the probability for the fast process to be exactly in $0$ as well. 

Beside functional Large Deviation principles, the analytic Singular Perturbation theory can provide an alternative approach to derive sharp asymptotics of the distibution of $\mathbf{X}^\f(\infty)$. This theory has been applied, in particular, in \cite{Yin13} to obtain asymptotics of the solutions of backward or forward Kolmogorov equations for jump processes with two-time scales; coupled queuing systems (\cite{Kne86bis, Kne95}, \cite{Schuss10} - Chap.9) have been also addressed in this framework. Specifically, an asymptotic expansion for the whole distribution of $\mathbf{X}^\f(\infty)$ on $\mathbb{N}^2$ of the form
$$
\mathbb{P}(\mathbf{X}^\f(\infty) = A \, \pmb{\xi}) = 
\frac{1}{A} 
\exp \left [ - A \cdot H(\pmb{\xi}) - h_0(\pmb{\xi}) + 
O \left ( \frac{1}{A} \right ) \right ], 
\qquad \pmb{\xi} = (x,y) \in \mathbb{R}_+^2, 
$$
is assumed to exist with large (a-dimensional) scaling parameter 
$A = \lambda^\f_\net/\theta$, and where real functions $H$, $h_0$ on 
$\mathbb{R}^2_+$ satisfy $H(\pmb{\xi}^*) = 0$ (with point $\pmb{\xi}^*$ as in Theorem \ref{thm:main}) together with smoothness properties. At the present stage, with no claim to formally justify the existence of such an expansion, these Singular Perturbation methods enable one to determine functions $H$ and $h_0$ explicitly giving, in particular, 
$$
H(\pmb{\xi}) = \Phi(x) + \Psi(y), \qquad \pmb{\xi} = (x,y),
$$
for simple functions $\Psi$ and $\Psi$; the latter relation thus provides another argument for the asymptotic independence (at logarithmic order) of the components of $\mathbf{X}^\f(\infty)$ discussed above. This Singular Perturbation framework for the estimation of the whole distribution of vector $\mathbf{X}^\f(\infty)$ and its justification is the object of current investigations \cite{SS20}.

\subsection{Model generalization}
In this paper, our goal was to initiate the analysis of a new class of stochastic model for mobile networks. The general idea of these models is to forget about keeping track of all users, but instead to focus on a subset of the whole network and take into account the rest of the network through a balance equation. Here, we focused as a first step on a single cell in equilibrium but more general situations can be considered.

Specifically, in the case of a single cell, we could for instance 
consider an ``imbalance'' parameter $\beta > 0$ and consider the balance equation
\[ 
\lambda^\f_\net = \beta \theta \cdot \E \left( X^\f_2(\infty) \right) 
\]
instead of~\eqref{eq:FP}. This new fixed-point equation would mean that the ratio of flows from and to the rest of the network is equal to $\beta$. Thus for $\beta > 1$, this would amount to considering a cell where more users enter than exit, and the opposite for $\beta < 1$. Of course, such an imbalance could not be sustained for the whole network but could hold locally. Studying what happens to the constrained model when enforcing this equation instead of~\eqref{eq:FP} constitutes an interesting research direction.

Finally, another way to generalize the model would be to consider several cells instead of only one. In this case, there are different flows from and to the rest of the network, as well as within the considered cells. The first difficulty to solve would be to find a relevant balance equation generalizing~\eqref{eq:FP}, which would probably be multi-dimensional. For instance, if one considers $n$ cells, there would now be potentially $2n + n + n(n-1)/2 + n$ parameters: one arrival rate per class and per cell, one capacity per cell, a mobility rate between each pair of cells and a mobility rate from each cell to the rest of the network.



\begin{thebibliography}{10}

\bibitem{Anton19:0}
E.~Anton, U.~Ayesta, and F.~Simatos.
\newblock On the impact of mobility in cellular networks.
\newblock In {\em WiOpt 19}, 2019.

\bibitem{Baynat15:0}
B.~Baynat, R.-M. Indre, N.~Nya, P.~Olivier, and A.~Simonian.
\newblock Impact of mobility in dense LTE-A networks with small cells.
\newblock In {\em Vehicular Technology Conference (VTC Spring), 2015 IEEE
 81st}, pages 1--5, May 2015.

\bibitem{Bonald04:0}
T.~Bonald, S.C.~Borst, and A.~Proutiere.
\newblock How mobility impacts the flow-level performance of wireless data
 systems.
\newblock In {\em Proc. INFOCOM '04}, volume~3, pages 1872--1881, March 2004.

\bibitem{Bonald09:1}
T.~Bonald, S.~Borst, N.~Hegde, M.~Jonckheere, and A.~Proutiere.
\newblock Flow-level performance and capacity of wireless networks with user mobility.
\newblock {\em Queueing Syst.}, 63(1-4):131--164, 2009.

\bibitem{Borst09:0}
S.C.~Borst, N.~Hegde, and A.~Proutiere.
\newblock Mobility-driven scheduling in wireless networks.
\newblock In {\em Proc. IEEE INFOCOM '09}, pages 1260--1268, 2009.

\bibitem{Borst06:0}
S.~Borst, A.~Proutiere, and N.~Hegde.
\newblock Capacity of wireless data networks with intra- and inter-cell
 mobility.
\newblock In {\em Proc. IEEE INFOCOM '06}, pages 1058--1069, 2006.

\bibitem{Borst13:0}
S.~Borst and F.~Simatos.
\newblock A stochastic network with mobile users in heavy traffic.
\newblock {\em Queueing Syst.}, 74(1):1--40, 2013.

\bibitem{Freidlin78:1}
M.I.~Freidlin.
\newblock The averaging principle and theorems on large deviations.
\newblock {\em Russian Math. Surveys}, 33(5):117--176, 1978.

\bibitem{Grossglauser01:0}
M.~Grossglauser and D.~Tse.
\newblock Mobility increases the capacity of ad-hoc wireless networks.
\newblock In {\em Proc. IEEE INFOCOM '01}, volume~3, pages 1360--1369, 2001.

\bibitem{Hajek82:0}
B.~Hajek.
\newblock Hitting-time and occupation-time bounds implied by drift analysis with applications.
\newblock {\em Advances in Applied Probability}, 14(3):502--525, 1982.

\bibitem{Huang16:0}
G.~Huang, M.~Mandjes, and P.~Spreij.
\newblock Large deviations for {M}arkov-modulated diffusion processes with
 rapid switching.
\newblock {\em Stochastic Process. Appl.}, 126(6):1785--1818, 2016.


\bibitem{Kne86bis} C.~Knessl, B.J.~Matkowsky, Z.~Schuss, C.~Tier, 
\newblock On the Performance of State-Dependent Single Server Queues, \newblock {\em SIAM Journal of Applied Mathematics}, Vol.46, No.4, Dec. 1986

\bibitem{Kne95} C.~Knessl, C.~Tier, 
\newblock Applications of Singular Perturbation Methods in Queueing, \newblock {\em Advances in Queueing Theory, Methods and Open Problems}, Probability and Stochastics Series, pp.311--336, CRC Press, 1995

\bibitem{Lin10:0}
M.~Lin, A.~Wierman, and B.~Zwart.
\newblock {The Average Response Time in a Heavy-traffic Srpt Queue}.
\newblock {\em SIGMETRICS Perform. Eval. Rev.}, 38(2):12--14, October 2010.

\bibitem{Liptser96:0}
Robert Liptser.
\newblock Large deviations for two scaled diffusions.
\newblock {\em Probab. Theory Related Fields}, 106(1):71--104, 1996.

\bibitem{Ma14:0}
H.~Ma, D.~Zhao, and P.~Yuan.
\newblock Opportunities in mobile crowd sensing.
\newblock {\em IEEE Communications Magazine}, 52(8):29--35, Aug 2014.

\bibitem{Olivier19:0}
P.~Olivier, A.~Simonian, and F.~Simatos.
\newblock Performance analysis of data traffic in small cells networks with user mobility.
\newblock In Antonio Puliafio and Kishor~S. Trivedi, editors, {\em Systems
 Modeling: Methodologies and Tools}, EAI/Springer Innovations in Communication and Computing, pages 177--193. EAI/Springer Innovations in Communication and Computing, 2019.

\bibitem{Puhalskii16:0}
A.A.~Puhalskii.
\newblock On large deviations of coupled diffusions with time scale separation.
\newblock {\em Ann. Probab.}, 44(4):3111--3186, 2016.

\bibitem{Rege96:0}
K.M.~Rege and B.~Sengupta.
\newblock Queue-length distribution for the discriminatory processor-sharing queue.
\newblock {\em Operations Research}, 44(4):653--657, 1996.

\bibitem{Robert03:0}
P.~Robert.
\newblock {\em Stochastic Networks and Queues}.
\newblock Stochastic Modelling and Applied Probability Series. Springer-Verlag, New York, 2003.
\newblock xvii+398 pp.

\bibitem{Schuss10} Z.~Schuss, \newblock Theory and Applications of Stochastic Processes, An analytical Approach, 
\newblock ed. Springer, {\em Applied Mathematical Sciences} Series, 
Vol.170, 2010

\bibitem{Simatos10:0}
F.~Simatos and D.~Tibi.
\newblock Spatial homogenization in a stochastic network with mobility.
\newblock {\em Ann. Appl. Probab.}, 20(1):312--355, 2010.

\bibitem{Simonian16:0}
A.~Simonian and P.~Olivier.
\newblock Performance of data traffic in small cells networks with inter-cell mobility.
\newblock In {\em Proc. VALUETOOLS '16}, 2016.

\bibitem{SS20} 
F.~Simatos, A.~Simonian. 
\newblock Heavy load analysis of the multi-class Processor-Sharing queue with impatience, 
\newblock In preparation.

\bibitem{Veretennikov99:0}
A.Yu.~Veretennikov.
\newblock On large deviations in the averaging principle for {SDE}s with a
 ``full dependence''.
\newblock {\em Ann. Probab.}, 27(1):284--296, 1999.

\bibitem{Veretennikov13:0}
A.Yu.~Veretennikov.
\newblock On large deviations in the averaging principle for {SDE}'s with a ``full dependence'', revisited.
\newblock {\em Discrete Contin. Dyn. Syst. Ser. B}, 18(2):523--549, 2013.

\bibitem{Yin13} G.G.~Yin, A.~Zhang, 
\newblock Continuous-Time Markov Chains and Applications, A two-time scale approach, 
\newblock ed. Springer, 2013

\end{thebibliography}


\end{document}